\documentclass[12pt]{amsart}
\usepackage{amssymb,amscd}

\let\frak\mathfrak

\def\>{\relax\ifmmode\mskip.666667\thinmuskip\relax\else\kern.111111em\fi}
\def\<{\relax\ifmmode\mskip-.333333\thinmuskip\relax\else\kern-.0555556em\fi}
\def\vsk#1>{\vskip#1\baselineskip}
\def\vv#1>{\vadjust{\vsk#1>}\ignorespaces}
\def\vvn#1>{\vadjust{\nobreak\vsk#1>\nobreak}\ignorespaces}

\let\Medskip\medskip
\def\medskip{\par\Medskip}
\let\Bigskip\bigskip
\def\bigskip{\par\Bigskip}

\let\Maketitle\maketitle
\def\maketitle{\Maketitle\thispagestyle{empty}\let\maketitle\empty}

\newtheorem{thm}{Theorem}[section]
\newtheorem{cor}[thm]{Corollary}
\newtheorem{lem}[thm]{Lemma}

\numberwithin{equation}{section}

\theoremstyle{definition}
\newtheorem*{rem}{Remark}

\let\mc\mathcal
\let\nc\newcommand

\let\al\alpha

\let\la\lambda

\let\phi\varphi
\let\si\sigma

\let\om\omega
\let\Om\Omega

\let\geq\geqslant

\let\leq\leqslant

\let\on\operatorname
\let\bi\bibitem
\let\bs\boldsymbol

\def\C{{\mathbb C}}
\def\Z{{\mathbb Z}}

\def\F{{\mc F}}
\def\L{{\mc L}}

\def\+#1{^{\{#1\}}}

\def\End{\on{End}}

\def\res{\on{res}}

\def\fb{\mathfrak{b}}
\def\fg{\mathfrak{g}}
\def\fh{\mathfrak{h}}
\def\fn{\mathfrak{n}}

\def\beq{\begin{equation}}
\def\eeq{\end{equation}}
\def\be{\begin{equation*}}
\def\ee{\end{equation*}}

\nc{\bea}{\begin{eqnarray*}}
\nc{\eea}{\end{eqnarray*}}
\nc{\bean}{\begin{eqnarray}}
\nc{\eean}{\end{eqnarray}}
\nc{\Ref}[1]{{\rm(\ref{#1})}}

\def\h{{\mathfrak h}}
\def\n{{\mathfrak n}}
\def\fsl{\mathfrak{sl}}

\def\CD{{\mathcal{D}}}

\let\ga\gamma
\let\Ga\Gamma

\nc{\Il}{{\mc I_{\bs\la}}}
\nc{\bla}{{\bs\la}}
\nc{\Fla}{\F_{\bs\la}}
\nc{\tfl}{{T^*\Fla}}
\nc{\GL}{{GL_n(\C)}}
\nc{\GLC}{{GL_n(\C)\times\C^*}}

\def\ka{{\kappa}}
\def\slt{{\frak{sl}_2}}
\def\OT{{\otimes_{a=1}^nM_{m_a}}}
\def\mm{{\bs m}}

\def\lra{\longrightarrow}

\begin{document}

\hrule width0pt
\vsk->

\hrule width0pt
\vsk->

\title[BGG resolutions via configuration spaces]
{BGG resolutions via configuration spaces }

\author[M.\,Falk, V.\,Schechtman, and A.\,Varchenko]
{M.\,Falk$\>^{\star,1}$, V.\,Schechtman$\>^{*}$,
and A.\,Varchenko$\>^{\diamond,2}$}

{\let\thefootnote\relax
\footnotetext{\vsk-.8>\noindent
${\>\kern-\parindent}^1\,$Supported in part by Fulbright U.S. Scholars grant\\
${\>\kern-\parindent}^2\,$Supported in part by NSF grant DMS-1101508}

\maketitle

\begin{center}
{\it $^\star\<$Department of Mathematics and Statistics\,--\, Northern Arizona University
\\ Flagstaff, AZ 86011, USA \/}

\medskip
{\it $^*\<$ Institut de Math\'ematiques de Toulouse\,--\,  Universit\'e Paul Sabatier\\ 118 Route de Narbonne,
31062 Toulouse, France \/}

\medskip
{\it $^\diamond\<$Department of Mathematics, University of North Carolina
at Chapel Hill\\ Chapel Hill, NC 27599-3250, USA\/}
\end{center}

\medskip\medskip

{\it To the memory of I.M.Gelfand, on the occasion of his centenary
(1913 - 2013)}

\medskip

\begin{abstract}

We study the blow-ups of configuration spaces. These spaces have a structure of what we
call an Orlik - Solomon manifold; it allows us to compute the
intersection cohomology of certain flat connections with logarithmic singularities using some Aomoto type complexes of logarithmic forms.
Using this construction we realize geometrically the $\slt$ Bernstein - Gelfand - Gelfand resolution as an Aomoto complex.

\end{abstract}

{\small \tableofcontents }

\section{Introduction}

Let us discuss briefly some general perspective and motivation.

\smallskip
\noindent
{\it Localization of $\fg$-modules: two patterns.}
${}$\\
(a) {\it Localization on the flag space.}
Let $\fg$ be a complex semisimple Lie algebra, $\fh\subset \fg$  a Cartan subalgebra whence the root system $R\subset \fh^*$; fix a
base of simple roots $\Delta \subset R$ whence a decomposition
$\fg = \fn_-\oplus \fh \oplus \fn_+$. The classical
Bernstein - Gelfand - Gelfand resolution is the left resolution
of a simple finite dimensional $\fg$-module $L_\chi$ of highest
weight $\chi - \rho$ (where $\rho$ is the half-sum of the positive roots) of the form
\bean
\label{BGG}
0 \lra C_n \lra \ldots \lra C_0 \lra L_\chi \lra 0 ,
\eean
where
\bea
C_i = \oplus_{w\in W_i} M_{w\chi},
\eea
cf. \cite{BGG}.
Here $M_\lambda$ denotes the Verma module of the highest weight
$\lambda - \rho$, $W_i\subset W$ is the set of elements of the
Weyl group of length $i$.

We can pass to contragradient duals  and use the isomorphism
$L_\chi = L_\chi^*$ given by the Shapovalov form to get
a right resolution
\bean
\label{BGG*}
0 \lra L_\chi \lra C_0^* \lra \ldots \lra C_n^* \lra 0
\eean
where
$$
C^*_i = \oplus_{w\in W_i} M^*_{w\chi}.
$$
A geometric explanation of the last complex was given by
Kempf, \cite{K}, who interpreted $\Ref{BGG*}$ as a Cousin
complex connected with the filtration of the flag space
$G/B$ by unions of Schubert cells ($G$ being a semisimple group
with Lie algebra $\fg$ and $B\subset G$ the Borel subgroup
with $Lie(B) = \fb := \fh\oplus \fn_+$) . Here the $i$-th term
is interpreted as a relative cohomology space
with support in the union of Schubert cells of codimension $i$.
This geometric picture is a part of Beilinson - Bernstein
theory which says that some reasonable category of $\fg$-modules
is equivalent to a category of (twisted) $\CD$-modules over
$G/B$, \cite{BB}.
\\
(b) {\it Localization on configuration spaces.}
In a different direction, contragradient Verma modules
and irreducible representations have been realized in \cite{SV}
in certain spaces of logarithmic differential forms
on configuration spaces. This may be upgraded to an equivalence
between  some category of $\fg$-modules and some category
of $\CD$-modules over configuration spaces, cf. \cite{KS, BFS, KV}.

\smallskip
\noindent
{\it Blow-ups and their "Schubert" stratifications.}
\\
In this note we propose a construction which provides
a geometric interpretation of the resolutions similar to the BGG resolution in \Ref{BGG*}.
 The main new idea is to use
the blow-ups of hyperplane arrangements
(in our case -- the configuration arrangements)
studied in \cite{ESV, STV, BG, V, DCP}. We define some natural
stratifications   on such blow-ups which play the role of the
Schubert stratification on $G/B$. On each stratum
we consider the Aomoto complex of  logarithmic Orlik-Solomon forms;
they are subcomplexes of the de Rham complexes of standard local systems from \cite{SV}. (In fact the stratification itself depends
on a local system).

This way we get double complexes with one differential
induced by the de Rham differential and the other one being
the residue. The residue differential gives rise to BGG-like complexes.
For the trivial local system we get the complexes
considered in \cite{BG}; in our case the combinatorics
of the ``Schubert stratification" depends on the
Cartan matrix and a finite number of dominant
weights.

We illustrate this construction for  $\fg = \fsl_2$. In this case
we obtain the BGG resolutions of tensor products of
finite dimensional $\fg$-modules, and the complex associated with our double complex
calculates
the intersection cohomology of the corresponding local system.

\smallskip
In Section \ref{Double complex of a filtered manifold}, we consider a complex analytic manifold $X$,
a divisor $D\subset X$ with normal crossings and a holomorphic flat connection on $X$. We construct a complex
which calculates the cohomology of $X$ with coefficients in the local system associated with the flat connection.

In Section \ref{Logarithmic residue complex of Orlik-Solomon forms}, we define an Orlik-Solomon manifold,
a flat connection with logarithmic singularities on an Orlik-Solomon manifold, the associated finite-dimensional
Aomoto complex. Theorem \ref{thm 4} says that the Aomoto complex calculates
the cohomology of the Orlik-Solomon manifold with coefficients in the local system associated with the connection.
Theorem \ref{thm 4} is our first main result.

In Section \ref{sec resol}, we discuss the minimal resolution of singularities of an arrangement.
In Section \ref{sec weights}, we introduce  weighted Orlik-Solomon manifolds associated with weighted arrangement of hyperplanes.
In Section \ref{Highe}, we review the definition of the BGG resolution for the Lie algebra $\slt$.
In Section \ref{sec slT}, we realize geometrically the $\slt$ BGG resolution as the skew-symmetric part
of the Aomoto complex of a suitable weighted Orlik-Solomon manifold. Theorem \ref{thm skew main} is our second
main result.
In Section \ref{BGG resolution and flag forms}, we discuss the relations between the BGG resolution and the complex of flag forms.
In Section \ref{Intersec}, we discuss the relations between the BGG resolution and intersection cohomology.

\smallskip
We thank A.\,Beilinson, V.\,Ginzburg, H.\,Terao for useful discussions and the Max Planck Institute for Mathematics for hospitality.

\section{Residue complex of a filtered manifold}
\label{Double complex of a filtered manifold}

\subsection{Local system of a flat connection}

\label{conn}
Let $X$ be a smooth connected complex analytic manifold.
Given a natural number $r$, let $\nabla$ be a holomorphic flat connection on the trivial bundle
$X\times \C^r\to X$. The
 sheaf $\L$ on $X$ of flat sections of $\nabla$ is a locally constant sheaf.
If $s$ is a differential form with values in $\C^r$, we denote  $d_\L s:= \nabla s = d s + \om \wedge s$ where
$\om$ is the connection form, a differential 1-form with values in $\End(\C^r)$. We have $d_\L^2=0$.

Let $(\Omega^\bullet_X\otimes\C^r, d_\L)$ be the de Rham complex of sheaves of $\C^r$-valued holomorphic differential forms on $X$
with differential $d_\L$. The cohomology $H^\bullet(X;\L)$ of $X$ with coefficients in $\L$
is canonically isomorphic to the hypercohomology $H^\bullet(X,\Omega_X^\bullet\otimes\C^r)$.

\subsection{Residue complex of sheaves}
\label{sec sheaves}

Let $D\subset X$ be a divisor with normal crossings. Namely, we assume that $X$ is covered by charts such that in each chart $D$
 is the union of several coordinate hyperplanes or the empty set. Such
charts are called {\it linearizing}.
 We define
\bea
\mc Z=\{X=Z_0\supset D=Z_1\supset Z_2\supset \dots\}
\eea
the associated filtration of $X$ by closed subsets  as follows. A  point $x\in X$ belongs to $Z_i$ if
in a linearizing chart $x$ belongs to the intersection of $i$ distinct coordinate hyperplanes of $D$. Thus
$\on{codim}_XZ_i=i$ if $Z_i$ is  nonempty. We denote by $C_{i,j},\, j=1,2,\dots,$ connected components of
$Z_i-Z_{i+1}$. Each $C_{i,j}$ is a smooth connected complex analytic submanifold of $X$ of codimension $i$. We set $C_{0,1}=X-D$.

Let  $\Omega^l_{C_{i,j}}$ be the sheaf of  holomorphic differential $l$-forms on $C_{i,j}$.
 Let $f: C_{i,j}\hookrightarrow X$ be the natural embedding
and $f_*\Omega^l_{C_{i,j}}$ the direct image sheaf.
 We denote
\bea
\Omega^l_{X,\mc Z}\,= \,\oplus_{i,j} \,f_*\Omega^{l-2i}_{C_{i,j}}.
\eea
Let $d_\L :f_*\Omega^l_{C_{i,j}}\otimes\C^r \to f_*\Omega^{l+1}_{C_{i,j}}\otimes\C^r$ be the differential of the connection $\nabla|_{C_{i,j}}$
and
\bea
\res : f_*\Omega^l_{C_{i,j}}\otimes\C^r \to f_*\Omega^{l-1}_{C_{i+1,j'}}\otimes\C^r
\eea
 the residue map, if $C_{i+1,j'}$ lies in the closure  $\overline{C_{i,j}}$, and  the zero map otherwise.
 The map $\tilde d =d_\L + \res$ defines the complex of sheaves on $X$,
\bea
 0 \to \Omega^0_{X,\mc Z}\otimes\C^r \xrightarrow{\tilde d} \Omega^{1}_{X,\mc Z}\otimes\C^r \xrightarrow{\tilde d} \Omega^{2}_{X,\mc Z}\otimes\C^r \xrightarrow{\tilde d} \dots\,.
\eea

The natural embeddings $  \Omega^l_X\otimes\C^r \hookrightarrow \Omega^l_{C_{0,1}}\otimes\C^r \hookrightarrow \Omega^l_{X,\mc Z}\otimes\C^r$ define
an injective homomorphism of complexes
\bean
\label{1st emb}
 (\Omega^\bullet_X\otimes\C^r, d_\L) \hookrightarrow
(\Omega^\bullet_{X,\mc Z}\otimes\C^r,\tilde d).
\eean

\begin{thm}
\label{thm 1}
Homomorphism \Ref{1st emb} is a quasi-isomorphism.

\end{thm}

\begin{proof} It is enough to check this statement locally on $X$. In that case we may assume that
$X = \{ z=(z_1,\dots,z_k)\in \C^k\ |\ |z|<1\}$ and $D$ is the union of several coordinate hyperplanes in $X$. For that example,
the statement  is checked by direct calculation.
\end{proof}

\subsection{Residue complex of global sections}
\label{Double complex of global sections}

Let $\Ga(C_{i,j},\Om^l_{C_{i,j}})$ be the space of global sections of $\Omega^l_{C_{i,j}}$.
Denote
\bea
\Ga^l(X,\mc Z;\C^r)\,=\, \oplus_{i,j}\, \Ga(C_{i,j},\Omega^{l-2i}_{C_{i,j}})\otimes\C^r.
\eea
The map $\tilde d =d_\L + \res$ defines the complex of vector spaces
\bea
 0 \to \Ga^0(X,\mc Z;\C^r) \xrightarrow{\tilde d} \Ga^{1}(X,\mc Z;\C^r) \xrightarrow{\tilde d} \Ga^{2}(X,\mc Z;\C^r) \xrightarrow{\tilde d} \dots\,.
\eea

\begin{thm}
\label{2d thm}
In addition to assumptions of Sections \ref{conn} and \ref{sec sheaves}, we assume  that for any $i,j$, the manifold $C_{i,j}$ is a Stein manifold.
Then there is the natural isomorphism $H^\bullet(X;\L)\simeq H^\bullet( \Ga^\bullet(X,\mc Z;\C^r),\tilde d)$.

\end{thm}

\begin{proof} For the Stein manifold $C_{i,j}$ the complex $(\Ga(C_{i,j},\Om^\bullet_{C_{i,j}})\otimes\C^r,d_\L)$ calculates $H^\bullet(C_{i,j};\L)$. This fact and
Theorem \ref{thm 1} imply Theorem \ref{2d thm}.
\end{proof}

\section{Logarithmic residue complex of Orlik-Solomon forms}
\label{Logarithmic residue complex of Orlik-Solomon forms}

\subsection{Affine arrangements}
\label{Affine arr}

Let $\mc A=\{H_i\}_{i\in  I}$ be an affine arrangement of hyperplanes, i.e., $\{H_i\}_{i\in I}$ is a finite nonempty
collection of
distinct hyperplanes in the affine complex space $\C^k$. Denote $U = \C^k - \cup_{i\in I}H_i$.  We denote
by $\Om^l_U$ the sheaf of holomorphic $l$-forms on $U$.

For any $i \in I$, choose a degree one polynomial function $f_i$ on $\C^k$ whose zero locus equals
 $H_i$. Define $\om_i= d \log f_i = df_i/f_i \in \Ga(U,\Om^1_U)$. Given a natural number $r$, we choose matrices
$P_i \in\End(\C^r)$, $i\in I$. Denote
\bea
\om = \sum_{i\in I} \om_i\otimes P_i \ \in \ \Ga(U, \Om^1_U)\otimes \End(\C^r).
\eea

The form $\om$ defines the connection $d + \om$ on the trivial bundle $U\times \C^r \to U$. We suppose
that  $d + \om$ is flat. Let $\L$ be the sheaf on $U$ of flat sections. Then $(\Om^\bullet_U\otimes\C^r,d_\L)$
is the complex of sheaves of $\C^r$-valued holomorphic differential forms on $U$ with differential $d_\L = d+\om$.

Define finite dimensional {\it Orlik-Solomon subspaces}
$A^p(\mc A) \subset \Ga(U,\Om^p_U)$
as the $\C$-linear subspaces generated by all forms $\om_{i_1}\wedge\dots\wedge\om_{i_p}$.
 Then the exterior multiplication by $\om$ defines the complex
\bea
 0 \to A^0\otimes\C^r \xrightarrow{\om} A^1\otimes\C^r \xrightarrow{ \om} A^2\otimes\C^r \xrightarrow{\om} \dots\,
\eea
as a subcomplex of $(\Ga(U,\Om^\bullet_U\otimes\C^r),d_\L)$. We call $(A^\bullet\otimes\C^r,\om)$ the {\it Aomoto
complex} of  $(U, d+\om)$.

Let $Y$ be any smooth compactification of $\C^k$ such that $H_\infty$ is a divisor. Write
$H = H_\infty\cup(\cup_{i\in I}H_i)$.  Then $U = Y-H$. (Typical examples for $Y$
include the complex projective space $\mc P^k$, $(\mc P^1)^k$ and any toric compactification of $\C^k$.) Note that $\om$
 can be uniquely extended to be an $\End(\C^r)$-valued rational 1-form $\om$ on $Y$.

\begin{thm} [\cite{ESV, STV}]
\label{ESV thm}

Suppose $\pi : X \to  Y$ is a blow-up of $Y$
with centers in H such that 1) $X$ is nonsingular, 2) $\pi^{-1}H$ is a normal crossing divisor,  3)
none of the eigenvalues of the residue of \,$\pi^*\om$ along any component of\, $\pi^{-1}H$ is a positive integer. Then the inclusion
$(A^\bullet\otimes\C^r,\om)\hookrightarrow (\Ga(U,\Om^\bullet_U)\otimes\C^r,d_\L)$ is a quasi-isomorphism.

\end{thm}

\begin{rem}
Assume that the pair $(X,\om)$ satisfies conditions 1) and 2) of Theorem \ref{ESV thm} but not condition 3). Then for almost all
$\ka\in\C^\times$, the pair $(X,\om/\kappa)$ satisfies all of the conditions 1)-3) of Theorem \ref{ESV thm}.

\end{rem}

\subsection{Orlik-Solomon manifolds}
\label{sec OS manifolds}

Let $X$ be a smooth connected complex analytic manifold, $\dim X=k$.
Let $D\subset X$ be a divisor with normal crossings and
$\mc Z=\{X=Z_0\supset D=Z_1\supset Z_2\supset \dots\}$
the associated filtration of $X$ by closed subsets.  We denote by $C_{i,j},\, j=1,2,\dots,$ connected components of
$Z_i-Z_{i+1}$ and set $C_{0,1}=X-D$.

Assume that for any $C_{i,j}$ we have:
\begin{enumerate}
\item[(i)]
An affine arrangement $\mc A_{i,j} = \{H_m\}_{m \in I_{i,j}}$ in $\C^{k-i}$ with complement
$U_{i,j} = \C^{k-i}-\cup_{m\in I_{i,j}} H_m$ and an
 analytic isomorphism $\phi_{i,j} : U_{i,j}\to C_{i,j}$.

\end{enumerate}
Assume that these objects have the following property.
\begin{enumerate}
\item[(ii)]
For any $i,j$, denote by $A^\bullet(U_{i,j})$
the Orlik-Solomon  spaces of $U_{i,j}$.
Let $C_{i+1,j'}$ lie in the closure $\overline{C_{i,j}}$ and
\bea
\res : \Ga(C_{i,j},\Omega^l_{C_{i,j}}) \to \Ga(C_{i+1,j'},\Omega^{l-1}_{C_{i+1,j'}})
\eea
 the residue map. Then  the image of $A^\bullet(U_{i,j})$
under the composition $(\phi_{i+1,j'})^*\circ\res\circ ((\phi_{i,j})^{-1})^*$
lies in $A^\bullet(U_{i+1,j'})$

\end{enumerate}
We say that $(X,D)$ is an {\it Orlik-Solomon manifold } if it has charts (i) with property (ii).

The images of Orlik-Solomon  spaces $A^\bullet(U_{i,j})$
under the isomorphism $\phi_{i,j}$ give
\linebreak
finite-dimensional subspaces of  $\Ga(C_{i,j},\Om^\bullet_{C_{i,j}})$.
We call these subspaces the {\it Orlik-Solomon  spaces of} $C_{i,j}$ and denote by $A^\bullet(C_{i,j})$.

\begin{rem}
Denote by $K=\{(0,1), ...\}$ the set of all pairs $(i,j)$ appearing as indices of components $C_{i,j}$ in the
decomposition of the pair $(X,D)$. Let $K_0\subset K$ be any subset which does not contain $(0,1)$.
 Denote $C_{K_0}\subset X$ the closure of $\cup_{(i,j)\in K_0}C_{i,j}$.
Denote $X_{K_0}=X-C_{K_0}, \, D_{K_0}=D-C_{K_0}$.  Then
$X_{K_0}$ is a smooth connected complex analytic  manifold and $D_{K_0}\subset X_{K_0}$ is a divisor with
normal crossings.
If $(X,D)$ is an Orlik-Solomon manifold, then $(X_{K_0}, D_{K_0})$ has the induced structure
of an Orlik-Solomon manifold.

\end{rem}

We describe examples of Orlik-Solomon manifolds in Section \ref{sec Examples of OS}.

\subsection{Aomoto complexes}
\label{sec Aomoto bicomplex}

Assume that $(X,D)$ is an Orlik-Solomon manifold and
$\nabla=d_\L=d+\om$ is  a holomorphic flat connection on   $X\times \C^r\to X$.
We say that $\nabla$ is a {\it flat connection with logarithmic singularities on the Orlik-Solomon
manifold} if the following property holds.
\begin{enumerate}
\item[(iii)]
For any $i,j$, the induced flat connection $\nabla_{i,j}:=(\phi_{i,j})^{*}\nabla$ on $U_{i,j}$ has the form
described in Section \ref{Affine arr}. Namely, $\nabla_{i,j} = d + \om_{i,j}$, where
\bea
\om_{i,j} = \sum_{m\in I_{i,j}} \om_m \otimes P_m
\eea
for suitable matrices $P_m\in\End(\C^r)$.

\end{enumerate}

If $\nabla$ is a flat connection with logarithmic singularities on the Orlik-Solomon
manifold $(X,D)$, then the exterior multiplication by $\om$ defines a finite-dimensional
complex
\linebreak
$(A^\bullet(C_{i,j})\otimes\C^r,\om)$ as a subcomplex of
 $(\Ga(C_{i,j},\Om^\bullet_{C_{i,j}})\otimes\C^r,d_\L=d+\om)$.

We denote
\bea
A^l(X,\mc Z;\C^r)\,= \,\oplus_{i,j} \,A^{l-2i}(C_{i,j})\otimes\C^r.
\eea
The map $\om+\res$ realizes the complex
\bea
 0 \to A^0(X,\mc Z;\C^r) \xrightarrow{\om+\res} A^1(X,\mc Z;\C^r) \xrightarrow{ \om+\res} A^2(X,\mc Z;\C^r) \xrightarrow{\om+\res} \dots\,
\eea
as a subcomplex of $( \Ga^\bullet(X,\mc Z;\C^r),\tilde d)$.

\begin{thm}
\label{thm 4}
Assume that  $\nabla = d+\omega$ is a flat connection with logarithmic singularities on the Orlik-Solomon
manifold $(X,D)$.
Assume that for any $i,j$, the form $\om_{i,j}$ on $U_{i,j}$ satisfies the conditions of Theorem \ref{ESV thm}
for a suitable resolution of singularities mentioned in Theorem \ref{ESV thm}. Then the embedding
$(A^\bullet(X,\mc Z;\C^r),\om+\res)\hookrightarrow ( \Ga^\bullet(X,\mc Z;\C^r),\tilde d)$ is a quasi-isomorphism.

\end{thm}

\begin{proof}

By Theorem \ref{ESV thm}, the embedding
$(A^\bullet(C_{i,j})\otimes\C^r, \om)\hookrightarrow (\Ga(C_{i,j},\Om^\bullet_{C_{i,j}})\otimes \C^r,d_\L)$
is a quasi-isomorphism. This implies Theorem \ref{thm 4}.
\end{proof}

\begin{cor}
\label{cor generic}
Assume that  $\nabla = d+\omega$ is a flat connection with logarithmic singularities on the Orlik-Solomon
manifold $(X,D)$.
For  $\ka\in\C^\times$, consider the flat connection $\nabla_\ka = d+\om/\ka$ and the associated
embedding $(A^\bullet(X,\mc Z;\C^r),\om/\ka+\res)\hookrightarrow ( \Ga^\bullet(X,\mc Z;\C^r),d+\om/\ka+\res)$.
Then for generic $\ka$ this embedding is a quasi-isomorphism.

\end{cor}

\section{Resolution of singularities of arrangements}
\label{sec resol}

\subsection{Minimal resolution of a hyperplane-like divisor}
Let $Y$ be a smooth connected complex analytic manifold and
$H$ a divisor. The divisor $H$ is {\it hyperplane-like}
if $Y$ can be covered by coordinate charts such that in each chart
$H$ is the union of hyperplanes. Such charts  are called {\it linearizing}.

Let $H$ be a hyperplane-like divisor, $V$
a linearizing chart. A {\it local edge} of $H$ in $V$ is
any nonempty irreducible intersection in $V$ of hyperplanes of $H$ in $V$.
A local edge  is {\it dense}  if the subarrangement of all hyperplanes in $V$ containing
the edge is irreducible: the hyperplanes cannot be partitioned into nonempty
sets so that, after a change of coordinates, hyperplanes in different
sets are in different coordinates. In particular, each hyperplane is a dense edge.
An {\it edge} of $H$ is the maximal analytic continuation in $Y$ of a local edge. An edge is called {\it dense} if it is locally dense. Any edge is an immersed submanifold in $Y$.
The irreducible components of $H$ are considered to be dense.

Let $H\subset Y$ be a hyperplane-like divisor.
Let $\pi : \tilde Y\to Y$ be the minimal resolution of singularities of $H$ in $Y$.
The minimal resolution is constructed by first blowing-up dense vertices of $H$,
then by blowing-up the proper preimages of dense
one-dimensional edges of $H$ and so on, see \cite{V, STV}.

\medskip

We have two basic examples of pairs $(Y,H)$.

\subsubsection{ Projective arrangement.}
\label{sec proj}
Let $\mc A =\{H_l\}_{l\in  I}$ be a  nonempty finite collection of
distinct hyperplanes in the complex projective space $\mc P^k$.
 Denote $H=\cup_{l\in I}H_l$.  Then $H\subset \mc P^k$ is a { hyperplane-like divisor}.
 Denote $U=\mc P^k-H$.

 For any $l,m\in I$,  we have $H_l-H_m=\on{div}(f_{l,m})$ for some rational function
 $f_{l,m}$ on $\mc P^k$. Define $\om_{l,m} =d\log f_{l,m}$. For $1\leq p\leq k$, we define the Orlik-Solomon space
 $A^p(U)$ as the $\C$-linear span of $\om_{l_1,m_1}\wedge\dots\wedge\om_{l_p,m_p}$.

 Given a natural number $r$, we choose matrices $P_l\in \End(\C^r)$, $l\in I$, such that $\sum_lP_l=0$.
 Fix $m\in I$ and define
 \bea
 \om =\sum_{l\in I}\om_{l,m}\otimes P_l.
 \eea
 The form $\om$ defines the connection $d+\om$ on $U\times \C^r\to U$. We call $d+\om$
a  {\it connection with logarithmic singularities}  on the complement  of the projective arrangement.

\subsubsection{ Discriminantal arrangement.}
\label{sec discr}

Let $Y=(\mc P^1)^k$. For $l=1,\dots,k$, we fix an affine coordinate $t_l$
on  the $l$-th factor of $Y$. For $1\leq l<m\leq k$, the subset
$H_{l,m}\subset Y$ defined by the equation $t_l-t_m=0$
is called a {\it diagonal} hyperplane. For $l,\,1\leq l\leq k$ and $z\in \C\cup\{\infty\}$,
the subset $H_{l}(z)\subset Y$ defined by the equation $t_l-z=0$
is called a {\it coordinate} hyperplane. If $z\in\C$ (resp. $z=\infty$), we call the coordinate hyperplane {\it finite}
(resp. {\it infinite}).

 A {\it discriminantal} arrangement in $Y$ is a finite collection
of diagonal and coordinate hyperplanes, which includes all infinite coordinate hyperplanes $H_l(\infty), l=1,\dots,k$, see \cite{SV}.
Define by $H$  the union of all of the hyperplanes of the arrangement. Then $H\subset Y$ is a hyperplane-like divisor.
Denote $U=Y-H$.

To every diagonal hyperplane $H_{l,m}$ we assign the 1-form $\om_{H_{l,m}}=d\log (t_l-t_m)$.
To every finite coordinate hyperplane $H_{l}(z)$  we assign the 1-form $\om_{H_l(z)}=d\log (t_l-z)$.
These are holomorphic forms on $U$. We define the Orlik-Solomon spaces
 $A^\bullet (U)$ as the graded components of the exterior $\C$-algebra generated by the 1-forms associated with
 the diagonal and finite coordinate hyperplanes.

Fix a natural number $r$. For any diagonal or finite coordinate hyperplane $H$ of the arrangement we
choose a matrix  $P_H\in \End(\C^r)$. Define
 \bea
 \om =\sum \om_H\otimes P_H,
 \eea
 where the sum is over all diagonal and finite coordinate hyperplanes of the discriminantal arrangement.
 This form $\om$ defines the connection $d+\om$ on the trivial bundle $U\times \C^r\to U$.
 We call $d+\om$
a  {\it connection with logarithmic singularities}  on  the complement  of the discriminantal arrangement.

\subsection{Examples of Orlik-Solomon manifolds}
\label{sec Examples of OS}
\subsubsection{Minimal resolution of singularities of a projective arrangement}
\label{res of proj}

Let $\mc A =\{H_l\}_{l\in  I}$ be a  projective arrangement of hyperplanes in $\mc P^k$.
 Denote $Y=\mc P^k$ and $H=\cup_{l\in I}H_l$.
 Let $\pi : \tilde Y\to Y$ be the minimal resolution of singularities of $H$ in $Y$
 and $\tilde H=\pi^{-1}H$. Then $\tilde H\subset \tilde Y$ is a divisor with normal crossings.
For the pair $(\tilde Y, \tilde H)$, we introduce components $C_{i,j} \subset \tilde Y$ as in Section
\ref{Double complex of a filtered manifold}.
It is clear from the construction of the minimal resolution that each $C_{i,j}$ is naturally isomorphic
to the complement of an affine arrangement and these isomorphisms have property (ii) of Section
\ref{sec OS manifolds}. Thus $(\tilde Y,\tilde H)$ has the {\it natural structure} of an Orlik-Solomon manifold.

\subsubsection{Minimal resolution of singularities of a discriminantal arrangement}
\label{res discr}

Let $\mc A =\{H_l\}_{l\in  I}$ be a discriminantal arrangement of hyperplanes in $(\mc P^1)^k$.
 Denote $Y=(\mc P^1)^k$ and $H=\cup_{l\in I}H_l$.
 Let $\pi : \tilde Y\to Y$ be the minimal resolution of singularities of $H$ in $Y$
 and $\tilde H=\pi^{-1}H$. Then $\tilde H\subset \tilde Y$ is a divisor with normal crossings.
For the pair $(\tilde Y, \tilde H)$, we introduce components $C_{i,j} \subset \tilde Y$ as in Section
\ref{Double complex of a filtered manifold}.
It is clear from the construction of the minimal resolution that each $C_{i,j}$ is natually isomorphic
to the complement of an affine arrangement and these isomorphisms have property (ii) of Section
\ref{sec OS manifolds}. Thus $(\tilde Y,\tilde H)$ has the {\it natural structure} of an Orlik-Solomon manifold.

\section{Weighted arrangements}
\label{sec weights}

\subsection{Weighted projective arrangement}
\label{Weighted projective arrangement}

Let $\mc A =\{H_l\}_{l\in  I}$ be a  projective arrangement of hyperplanes in $Y=\mc P^k$.
 Denote $H=\cup_{l\in I}H_l$,\  $U=Y-H$.

The arrangement $\mc A$ is {\it weighted} if a map $a: I\to \C,\ l\mapsto a_l,$ is given
such that $\sum_{l\in I}a_l =0$. The number $a_l$ is called the {\it weight} of $H_l$.
Let $X_\al$ be an edge of $\mc A$. Denote $I_\al = \{l\in I\ |\ H_l\supset X_\al\}$.
The number $a_\al = \sum_{l\in I_\al}a_l$ is called the {\it weight } of $X_\al$.
The edge $X_\al$ is  {\it resonant} if $a_\al=0$.

 Fix $m\in I$ and define
 \bea
 \om_a =\sum_{l\in I}\om_{l,m}\otimes a_l,
 \eea
 see Section \ref{sec proj}.
 The form $\om_a$ defines  the flat connection $d+\om_a$ on $U\times \C\to U$. We call $d+\om_a$
the  {\it connection associated with weights} $a$.

Let $\pi :\tilde Y\to Y$ be the minimal resolution of singularities of $H$. Denote $\tilde H=\pi^{-1}H$.
The irreducible components of $\tilde H$ are labeled by dense edges $X_\al$ of $H$. Such a component will be denoted by
$\tilde H_\al$.
Consider $(\tilde Y, \tilde H)$ with its natural structure of an Orlik-Solomon manifold, see Section \ref{res of proj}.

Denote $\tilde \om_a = \pi^{*}\om_a$. The form $\tilde \om_a$ is regular on an irreducible component of $\tilde H$
if and only if the corresponding dense edge of $H$ is resonant.

Let $J$ be  the set of all nonresonant dense edges of $H$ and $\tilde J$ any set of dense edges such that $J\subseteq \tilde J$.
Denote $\tilde H_{\tilde J}=\cup_{X_\al \in \tilde J} \tilde H_\al$, $X=\tilde Y - \tilde H_{\tilde J}$, $D=\tilde H - \tilde H_{\tilde J}$.
Then $(X,D)$ is the Orlik-Solomon manifold with respect to the structure induced from $(\tilde Y, \tilde H)$, see
 Section \ref{sec OS manifolds}. The form
 $\tilde \om _a$ is regular on $X$ and $d+\tilde \om_a$ is a flat connection with logarithmic singularities on the Orlik-Solomon manifold $(X,D)$.
 Thus we may construct the associated complex $(A^\bullet(X,\mc Z),\tilde\om_a+\res)$ and
apply  Theorem \ref{thm 4} and Corollary \ref{cor generic} to the triple $(X,D, d+\tilde\om_a)$.
The complex $(A^\bullet(X,\mc Z),\tilde\om_a+\res)$ will be called the {\it Aomoto complex of the weighted Orlik-Solomon manifold} $(X,D)$.

\subsection{Weighted discriminantal arrangement}
\label{Weighted discr arrangement}

Let $\mc A =\{H_l\}_{l\in  I}$ be a discriminantal arrangement of hyperplanes in $Y=(\mc P^1)^k$.
 Denote  $H=\cup_{l\in I}H_l$, $U=Y-H$.

According to the definition in Section \ref{sec discr}, the discriminantal arrangement
 contains the
infinite coordinate hyperplanes $H_p(\infty), p=1,\dots, k$. Let $I_{\rm fin}\subset I$
be the set of indices of the remaining hyperplanes of $\mc A$.

The discriminantal arrangement $\mc A$ is {\it weighted} if a map $a: I_{\rm fin}\to \C,\ l\mapsto a_l,$ is given.
 The number $a_l$ is  the {\it weight} of $H_l, l\in I_{\rm fin}$. We also write $a(H_l):= a_l$.

We extend this map to the map $a:I\to \C$ as follows. We set the weight of an
infinite coordinate hyperplane $H_p(\infty)$ to be the number $-\sum a_q$ where
the sum is over all $q\in I_{\rm fin}$ such that $H_q$ is of the form $t_p-t_i=0$ for some $i$ or of the form
$t_p-z=0$ for some $z\in\C$.

Let $X_\al$ be an edge of $\mc A$. Denote $I_\al = \{l\in I\ |\ H_l\supset X_\al\}$.
The number $a_\al = \sum_{l\in I_\al}a_l$ is  the {\it weight } of $X_\al$.
The edge $X_\al$ is  {\it resonant} if $a(X_\al)=0$.

We define
 \bea
 \om_a =\sum_{l\in I_{\rm fin}}\om_{H_l}\otimes a_l,
 \eea
 see Section \ref{sec discr}.
 The form $\om_a$ defines the flat connection $d+\om_a$ on $U\times \C\to U$. We call $d+\om_a$
the  {\it connection associated with weights} $a$.

Let $\pi :\tilde Y\to Y$ be the minimal resolution of singularities of $H$. Denote $\tilde H=\pi^{-1}H$.
The irreducible components of $\tilde H$ are labeled by dense edges $X_\al$ of $H$. Such a component component will be denoted by
$\tilde H_\al$.
Consider $(\tilde Y, \tilde H)$ as the Orlik-Solomon manifold, see Section \ref{res discr}.

Denote $\tilde \om_a = \pi^{*}\om_a$. The form $\tilde \om_a$ is regular on an irreducible component of $\tilde H$
if and only if the corresponding dense edge of $H$ is resonant.

Let $J$ be the  set of all nonresonant dense edges of $H$ and $\tilde J$ any subset of dense edges such that $J\subseteq\tilde J$.
Denote $\tilde H_{\tilde J}=\cup_{X_\al \in \tilde J} \tilde H_\al$, $X=\tilde Y - \tilde H_{\tilde J}$, $D=\tilde H - \tilde H_{\tilde J}$.
Then $(X,D)$ is the Orlik-Solomon manifold with respect to the structure induced form $(\tilde Y, \tilde H)$, see
 Section \ref{sec OS manifolds}. The form
 $\tilde \om _a$ is regular on $X$ and $d+\tilde \om_a$ is a flat connection with logarithmic singularities on the Orlik-Solomon manifold $(X,D)$.
 Thus we may construct the associated complex $A^\bullet(X,\mc Z,\tilde\om_a+\res)$ and
apply  Theorem \ref{thm 4} and Corollary \ref{cor generic} to the triple $(X,D, d+\tilde\om_a)$.
The complex $A^\bullet(X,\mc Z,\tilde\om_a+\res)$ will be called the {\it Aomoto complex of the weighted Orlik-Solomon manifold} $(X,D)$.

\section{Highest weight representations of $\slt$}
\label{Highe}

\subsection{Modules}
\label{slt sec}

Consider the complex Lie algebra $\slt$ with standard basis $e,f,h$ such that
$[e,f]=h, [h,e]=2e, [h,f]=-2f$. We have $\slt=\n_-\oplus\h\oplus\n_+$,
where $\n_-=\C f, \h=\C h,$
\linebreak
$\n_+=\C e$.

Let $V$ be an $\slt$-module. For $\la\in \C$, let
$V[\la]=\{v\in V\ |\ hv=\la v\}$ be the subspace of weight $\la$.
Assume that $V$ has weight decomposition $V=\oplus_\la V[\la]$ with finite-dimensional spaces $V[\la]$. Then
the {\it restricted dual} of $V$ is $V^* := \oplus_\la V[\la]^*$.
The restricted dual has the {\it contragradient} $\slt$-module structure: for $\phi \in V^*$,
we have $\langle e\phi, v\rangle = \langle \phi, fv\rangle$,
$\langle f\phi, v\rangle = \langle \phi, ev\rangle$,
$\langle h\phi, v\rangle = \langle \phi, hv\rangle$.  We have $V[\la]^*=V^*[\la]$ for any $\la$.

For the Lie algebra $\n_-$ and a module $V$ we denote $C_\bullet (\n_-,V)$ the standard complex of $\n_-$ with coefficients
in $V$,
\bea
0\to C_{0}(\n_-,V) \to C_1(\n_-,V) \to 0,
\eea
where $C_{0} (\n_-,V)=\n_-\otimes V$, $C_1 (\n_-,V)=V$, and the map is $f\otimes v\mapsto fv$.
We have the weight decomposition
\bea
C_\bullet (\n_-,V) = \oplus_\la C_\bullet (\n_-,V)[\la],
\eea
where $C_\bullet (\n_-,V)[\la]$ is
\bean
\label{C(n,V)}
0\to\n_-\otimes V[\la+2]\to V[\la]\to 0.
\eean

\subsection{Verma modules}
\label{Verma sec}

For $m \in \mathbb{C}$, the Verma module $M_m$ is the
infinite dimensional $\mathfrak{sl}_2$-module
generated by a single vector
$v_m$ such that $hv_m = m v_m$ and $ev_m =0$. The
vectors $f^jv_m$, $j=0,1,\dots,$ form a basis of  $M_m$.
The action is given by the formulas
\bea
f \cdot f^jv_m  =  f^{j+1}v_m ,
\quad
   h \cdot f^jv_m  =  (m-2j)f^jv_m ,
 \quad
   e \cdot f^jv_m  =  j(m-j+1)f^{j-1}v_m .
\eea
Consider the contragradient module $M_m^*$ with the basis $\phi^{j}_{m}$, $j\in\Z_{\geq 0}$,
dual to the basis $f^jv_m$ of $M_m$. We have
\bea
f \cdot \phi^{j}_{m}  =  (j+1)(m-j) \phi^{j+1}_{m} ,
\quad
   h \cdot \phi^{j}_{m}  =  (m-2j)\phi^{j}_{m} ,
\quad  e \cdot \phi^{j}_{m}  =  \phi^{j-1}_{m} .
\eea
The Shapovalov symmetric bilinear form on $M_m$ is defined by the conditions
\bea
S(v_m,v_m) = 1,  \qquad S(fx,y)=S(x,ey),
\eea
for all $x,y\in M_m$. The Shapovalov form defines the morphism of modules
\bea
S : M_m\to M_m^*, \qquad x \mapsto S(x,\cdot).
\eea
The image  $L_m:=\on{Im}(S) \hookrightarrow M^*_m $ is irreducible.

If $m \notin \mathbb{Z}_{\geq 0}$, then $M_m$ is irreducible,
otherwise the subspace with basis $ f^{j}v_m, j\geq m+1,$ is
 a submodule which is identified with the Verma module $M_{-m-2}$
 under the map $M_{-m-2} \hookrightarrow M_m$, $f^jv_{-m-2} \mapsto f^{j+m+1}v_m$.
 The quotient $M_m/M_{-m-2}$ is an irreducible module with
 basis induced by $v_m, fv_m, \dots, f^mv_m$. The submodule $M_{-m-2} \hookrightarrow M_m$ is
 the kernel of the Shapovalov form.
  The induced Shapovalov form on  $M_m/M_{-m-2}$ identifies $M_m/M_{-m-2}$ and  $L_m \hookrightarrow M^*_m$.

We have the exact sequence of $\slt$-modules
\bea
0 \to L_m \to M_{m}^* \to M_{-m-2}^* \to 0 ,
\eea
which is called the {\it BGG resolution} of the irreducible $\slt$-module $L_m$, see \cite{BGG}.
We will keep two terms of this sequence
\bean
\label{compl 1}
 M_{m}^* \xrightarrow{\iota} M_{-m-2}^*
\eean
in which the epimorphism is denoted by $\iota$. We consider this map as a complex with terms in degree 0 and 1.

\subsection{Tensor product of Verma modules}
\label{product Verma sec}

For a vector $\mm = (m_1,\dots,m_n) \in \mathbb{C}^n$, denote
$|\mm| = m_1 + \dots + m_n$. Consider the tensor product $\OT$ of Verma modules.
 For
$J=(j_1,\dots,j_n) \in \mathbb{Z}_{\geq 0}^n$, let
\be
f^Jv_\mm := f^{j_1}v_{m_1}\otimes \dots \otimes f^{j_n}v_{m_n}.
\ee
The vectors $f^Jv_\mm$ form a basis of
$\OT$. We have
 \bea
 &&
f \cdot f^Jv_\mm = \sum_{a=1}^n f^{J+1_a}v_\mm,
\quad
h \cdot f^Jv_\mm = ( |\bs m|-2|J|) f^Jv_\mm,
\\
&&
\phantom{aaaaaaaa}
e \cdot f^Jv_\mm = \sum_{a=1}^n j_a (m_a-j_a+1) f^{J-1_a}v_\mm ,
\eea
where $J \pm 1_a = (j_1, \dots , j_a \pm 1, \dots , j_n)$.

We have the weight decomposition
\bea
\OT =\oplus_{k=0}^\infty \,(\OT)[|\bs m|-2k].
\eea
The basis in $(\OT)[|\bs m|-2k]$ is formed by the monomials $f^Jv_\mm$ with $|J|=k$.

Consider the restricted dual space $(\OT)^*$ with the weight decomposition
\bea
(\OT)^* =\oplus_{k=0}^\infty \,(\OT)^*[|\bs m|-2k].
\eea
The basis of $(\OT)^*[|\bs m|-2k]$ is formed by vectors
\bea
\phi^J_\mm : = \phi^{j_1}_{m_1}\otimes\dots\otimes \phi^{j_n}_{m_n}
\eea
with $|J|=k$.

The $\slt$-action is given by the formulas
\bea
f\cdot \phi^J_\mm = \sum_{a=1}^n (j_1+1)(m-j_a)\phi^{J+1_a}_\mm,
\quad
h\cdot \phi^J_\mm = (|\bs m|-2|J|) \phi^{J+1_a}_\mm,
\quad
e\cdot \phi^J_\mm = \sum_{a=1}^n \phi^{J-1_a}_\mm.
\eea

\subsection{Tensor product of complexes}
\label{Tensor product of complexes}

Let coordinates of $\mm=(m_1,\dots,m_n)$ be positive integers.
For $a=1,\dots,n$,
denote  by $A^0_{m_{a}} \xrightarrow{\iota_{a}} A^1_{m_a}$  the complex
$  M_{m_a}^* \xrightarrow{\iota_{a}} M_{-m_a-2}^* $.
Consider the tensor product $(A^\bullet_\mm, \iota)$ of these complexes, where
\bea
A^i_\mm = \oplus_{i_1+\dots+i_n=i}\, A^{i_1}_{m_1}\otimes\dots\otimes  A^{i_n}_{m_n}, \qquad
i=0,\dots,n,
\eea
with differential
\bea
\iota : x_1\otimes\dots\otimes x_n \mapsto \sum_{a=1}^n (-1)^{\deg x_1+\dots +\deg x_{a-1}} x_1\otimes \dots \otimes \iota_{a}x_a\otimes
\dots\otimes x_n.
\eea
The differential is a morphism of $\slt$-modules.
We have
\bea
\otimes_{a=1}^n L_a = \ker ( \iota : A^0_\mm \to A^1_\mm).
\eea
At all other degrees the complex $(A^\bullet_\mm, \iota)$ is acyclic.
Thus  $(A^\bullet_\mm, \iota)$ gives the resolution of $\otimes_{a=1}^n L_a$ which we will call the {\it BGG resolution} of $\otimes_{a=1}^n L_a$.

Consider the complex $C_\bullet(\n_-, A^\bullet_\mm)$,
\bean
\label{Lie product}
\n_-\otimes A^\bullet_\mm \xrightarrow{f} A^\bullet_\mm .
\eean
The differential $f$ of this complex commutes with the differential $\iota$ acting on $A^\bullet_\mm$.
Consider the complex $(B^\bullet_\mm, \tilde d)$, where
\bea
B^i_\mm = (\n_-\otimes A^{i}_\mm) \oplus A^{i-1}_\mm, \qquad i=0,\dots, n+1,
\eea
and
\bea
\tilde d \ :\  f\otimes x + y \ \mapsto \ fx - f\otimes \iota x + \iota y.
\eea
The embeddings
\bean
\label{emb}
&&
\\
&&
\notag
\n_-\otimes\otimes_{a=1}^nL_{m_a}  \hookrightarrow \n_-\otimes(\OT)^* = B^0_\mm,
\qquad
\otimes_{a=1}^nL_{m_a}  \hookrightarrow (\OT )^* \hookrightarrow B^1_\mm
\eean
define the morphism
of complexes
\bean
\label{bgg}
C_\bullet(\n_-,\otimes_{a=1}^nL_{m_a}) \to (B^\bullet_\mm,\tilde d).
\eean

\begin{lem}
This morphism is a quasi-isomorphism.
\end{lem}

This quasi-isomorphism will be called the {\it BGG resolution of}   $C_\bullet(\n_-,\otimes_{a=1}^nL_{m_a})$.

The complex $(B^\bullet_\mm,\tilde d)$ has weight decomposition. For any $\la\in\C$ we have
\bea
B^i_\mm[\la] = (\n_-\otimes A^{i}_\mm[\la+2]) \oplus A^{i-1}_\mm[\la],
\eea

In the next section we identify the complex $(B^\bullet_\mm[|\bs m|-2k], \tilde d)$
with the skew-symmetric part of the Aomoto complex of a suitable weighted Orlik-Solomon manifold.

\section{Discriminantal arrangements with $\slt$ weights}
\label{sec slT}

\subsection{Weighted discriminantal arrangement in $\C^k$}
\label{sec discarr}

Fix $\mm=(m_1,\dots,m_n)$ with positive integer coordinates and a positive integer $k$.
We assume that $m_j\leq k-1$ for $j=1,\dots, n_0$ and $m_j>k-1$ for $j=n_0+1,\dots,n$.
Fix $(z_1,\dots,z_n)\in\C^n$ with distinct coordinates. Fix a generic nonzero number $\ka$.

Consider $\C^k$ with coordinates $t_1,\dots,t_k$ and the weighted discriminantal arrangement
$\mc A$ consisting of the following hyperplanes:
$H_{i,j}$ defined by the equation $t_i-t_j=0$ for $1\leq i<j\leq k$,
$H_i^j$ defined by the equation $t_i-z_j=0$ for $i=1,\dots,k,\,j=1,\dots,n$.
The weights are $a_{i,j} = 2/\ka$, $a_i^j=-m_j/\ka$.
We denote by $H\subset \C^k$ the union of all hyperplanes of $\mc A$. Set
 $U=\C^k- H$.

The symmetric group $S_k$ acts on $\C^k$ by permutation of coordinates. The action preserves
the weighted  arrangement $\mc A$.

For $j=1,\dots,n_0$, let $I\subset \{1,\dots,n\}$ be a subset with $m_j+1$ elements.
The edge $X_I^j$ of $\mc A$ defined by equations
$t_i=z_j$ for $ i\in I$, is resonant.

\begin{lem}
The edges $X_I^j$, $j=1,\dots,n_0$, $|I|=m_j+1$, are the only resonant dense edges of $\mc A$.
\end{lem}

\begin{proof}
The dense edges of $\mc A$ have the form $t_{i_1}= \dots = t_{i_\ell}$ or
$t_{i_1}= \dots = t_{i_\ell}=z_j$ for $2\leq \ell \leq k$. One checks that the edges
of the former type are not resonant, and edges of the latter type are resonant if and
only if $\ell=m_j+1\leq k$.
\end{proof}

\subsection{Skew-symmetric part of Aomoto complex of $U$}
\label{sec twodim}

 The symmetric group $S_k$ naturally acts on the Orlik-Solomon spaces $A^\bullet(U)$.
 The {\it skew-symmetrization} of a form $\eta \in A^\bullet(U)$
is the form $\on{Skew}\eta := \sum_{\sigma\in S_k} (-1)^{|\si|}\si\eta$. The form $\on{Skew}\eta$ is skew-symmetric.
More generally, if $G\subset S_k$ is a subgroup, then
the $G$-{\it skew-symmetrization} of a form $\eta \in A^\bullet(U)$
is the form $\on{Skew}_G\eta := \sum_{\sigma\in G} (-1)^{|\si|}\si\eta$.

The skew-symmetric part $A^\bullet_-(U)$ of the Orlik-Solomon spaces $A^\bullet(U)$ is described in \cite{SV}. We have
$A^p_-(U)\ne 0$ only if $p=k-1,k$. Let $J=(j_1,\dots,j_n)$ be a vector with nonnegative integer
coordinates and $|J|=k$.
Define  $l_0(J)=0$ and $l_i(J)=j_1+\dots+j_i$ for $i=1,\dots,n$, and
\bea
\eta_{J,i}= d\log(t_{l_{i-1}(J)+1} - z_i)\wedge \dots \wedge d\log(t_{l_i(J)}-z_i)
\eea
for $i=1,\dots,n$.
Let $\omega_J$ be the skew-symmetrization of the $k$-form $\al_J \, \eta_{J,1}\wedge\dots\wedge\eta_{J,n}$, where
$\al_J=(\ka^{|J|}j_1!\dots j_n!)^{-1}$.

Let $J=(j_1,\dots,j_n)$ be a vector with nonnegative integer
coordinates and $|J|=k-1$.
Define  the $(k-1)$-form
$\eta_J=\al_J \, \eta_{J,1}\wedge\dots\wedge\eta_{J,n}$ as above,
and then $\om _J$ as the skew-symmetrization of $(-1)^k\eta_J$.

\begin{lem}
[\cite{SV}]
\label{lem skeW forms}
The forms $\{\om_J\}_{|J|=k}$ form a basis of $A^{k}_-(U)$.
The forms $\{\om_J\}_{|J|=k-1}$ form a basis of $A^{k-1}_-(U)$.

\end{lem}

Define the form
\bean
\label{oma}
\om_a = \sum_{H\in \mc A}\, a_H \,d\log f_H \ \in \ A^1(U).
\eean
The form $\om_a$ is symmetric with respect to the $S_k$-action.

\begin{lem}[\cite{SV}]
\label{Lem SV}
For any vector $\bs m\in \C^n$, the complex \ ${}$\ $\wedge \, \om_a : A^{k-1}_-(U) \to A^{k}_-(U)$ is isomorphic
to the weight component of weight $|\bs m|-2k$ of the complex $\n_-\otimes (\OT)^* \to (\OT)^*$.
The isomorphism sends $\om_J$ to $f\otimes \phi^J_\mm $ if $|J|=k-1$ and to $\phi^J_\mm$ if $|J|=k$.
\label{thm SV}
\end{lem}

\subsection{Skew-symmetric forms on $\mc P^m$}
\label{sec localproj}

For a positive integer $m$, consider a subset $I=\{1\leq i_0<\dots<i_{m}\leq k\}$
and the space $\C^{m+1}$ with coordinates $t_i, i\in I$. Consider
 the central arrangement in $\C^{m+1}$ consisting of coordinate hyperplanes and all diagonal hyperplanes.
This arrangement is preserved by the action of the  symmetric group $S_{m+1}$ which permutes  the coordinates.

Consider the projectivization  in $\mc P^m$ of the initial arrangement. The functions $u_{i_l} = t_{i_l}/t_{i_{0}}$,
$l=1,\dots,m$
are coordinates on an affine chart on $\mc P^m$. In these coordinates the projectivization of the initial arrangement  consists of hyperplanes
$u_{i_{l}}=0, u_{i_{l}}-1=0, u_{i_{l}}-u_{i_{q}}=0$ and the hyperplane at infinity.
 Denote $U\subset \mc P^m$ the complement to the arrangement.
 Let $A^\bullet_-(U)$ denote the skew-symmetric part of the Orlik-Solomon space $A^\bullet_-(U)$
with respect to the  $S_{m+1}$-action.

\begin{lem}
$A^p_-(U)=0$ if $p\ne m$, $\dim A^m_-(U)=1$. The form
\bea
\mu_I = d\log u_{i_{1}}\wedge\dots\wedge d\log u_{i_{m}}
\eea
generates  $A^m_-(U)$.
\label{lem onedim}
\end{lem}

\begin{proof}
Let $\tilde U$ denote the complement of the original central arrangement in $\C^{m+1}$.
The skew-symmetric part of $A^\bullet(\tilde U)$ is two-dimensional, $\dim A^p_-(\tilde U)=1$ for $p=k,k+1$.
The skew-symmetrizations of
\bea
\eta_{I,m} = d\log t_{i_{1}}\wedge\dots\wedge d\log t_{i_{m}}
\qquad \text{and} \qquad
\eta_{I} = d\log t_{i_{0}}\wedge\dots\wedge d\log t_{i_{m}}
\eea
form a basis in $A^\bullet_-(\tilde U)$.

Using the identity $d\log u_{i_{l}}=d\log t_{i_{l}}-d\log t_{i_{0}}$, one identifies $A^\bullet(U)$
with a subspace of the Orlik-Solomon space $A^\bullet(\tilde U)$ of the initial central arrangement
in $\C^{m+1}$. By
\cite[\S 6.1]{Dimca}, contraction along the Euler vector vector field
$\epsilon=\sum_{l=0}^{m} t_{i_{l}} \frac{\partial}{\partial t_{i_{l}}} $ defines an epimorphism
$\partial \colon A^\bullet(\tilde U) \to A^\bullet(U)$,
which restricts to an epimorphism $A^\bullet_-(\tilde U) \to A^\bullet_-(U)$
of skew-symmetric forms. The map $\partial$ is the boundary map in the acyclic
complex studied in \cite[\S 3.1]{OT}, and also coincides with the
residue map along the exceptional divisor in the blow-up of $\C^{m+1}$ at the origin.

Under this identification, the skew-symmetrization of the form $\eta_{I,m}$ equals a nonzero multiple
of the form  $\mu_I$ considered as an element of $A^\bullet(\tilde U)$. The  form $\eta_{I}$ is skew-symmetric and its contraction
along $\epsilon$  equals $\mu_I$. The contraction of  $\mu_I$ along $\epsilon$ is trivial since  $\partial^2=0$.
Then $A_-^p(U)=0$ for $p\ne m$ and $A_-^m(U)$ is spanned by $\mu_I$.
\end{proof}

\subsection{Weighted Orlik-Solomon manifold}
\label{associated weighted Orlik-Solomon}

Consider the minimal resolution  $\pi :\tilde Y\to \C^k$ of singularities of $H$, see
Section \ref{sec discarr}. Denote $\tilde H=\pi^{-1}H$.
The irreducible components of $\tilde H$ are labeled by dense edges of $H$.
We   denote by $X$ the manifold obtained from $\tilde Y$ by deleting the union of all irreducible components of $H$
corresponding to nonresonant dense edges. We set $D=\tilde H\cap X$. Then $D \subset X$ is a divisor with normal crossings
and $(X,D)$ is a weighted Orlik-Solomon manifold, see Sections \ref{Weighted projective arrangement} and \ref{Weighted discr arrangement}.
The symmetric group $S_k$ acts on the Orlik-Solomon manifold $(X,D)$. The action  preserves the weights.

Let $\mc Z=\{X=Z_0\supset D=Z_1\supset Z_2\supset \dots\}$ be the associated filtration by closed subsets, and $U=Z_0-Z_1=X-D$.

The irreducible components of $D$ are labeled by resonant dense edges of $H$.
For $j\in\{1,\dots,n_0\}$ and $I\subset\{1,\dots,n\}$, $|I|=m_j+1$, we denote by $\tilde H_I^j$ the component corresponding to the resonant dense edge $X_I^j$.
We denoted by $C_I^j$ the connected component of $Z_1-Z_2$ whose closure is $\tilde H_I^j$.
Then $C_I^j$ is isomorphic to the complement of the product of weighted arrangements in
$\mc P^{m_j}\times \C^{k-m_j-1}$, with weights induced by $\mc A$.
If $I=\{1\leq i_0<\dots<i_{m_j}\leq k\}$, then  $u_{i_l}$, $l=1,\dots,m_j$,
are coordinates on an affine chart on $\mc P^{m_j}$, see Section \ref{sec localproj}.
 The arrangement in
$\mc P^{m_j}$ has hyperplanes $u_{i_{l}}=0, u_{i_{l}}-1=0, u_{i_{l}}-u_{i_{q}}=0$ and the hyperplane at infinity.
 The weights induced by $\mc A$
are $-m_j/\ka$ for $u_{i_l}=0$ and $2/\ka$ for $u_{i_l}-1=0$ and $u_{i_l}-u_{i_q}=0$.
Coordinates on $\C^{k-m_j-1}$ are $t_i$, $i\in\{1,\dots,n\}-I$.
The arrangement in $\C^{k-m_j-1}$ is the discriminantal arrangement with hyperplanes
$t_i-t_q=0$, $i,q  \in \{1, \dots, k\}-I$ and $t_i-z_l=0$, $i\in \{1, \dots, k\}-I$,
$l\in\{1,\dots,n\}$. The weights of this arrangement in   $\C^{k-m_j-1}$ induced from $\mc A$ are given by the pair $(\bs m^j,\ka)$, where $\bs m^j=(m_1,\dots,-m_j-2,\dots,m_n)$,
see Section \ref{sec discarr}.

The set $\{ C_I^j\}_{j,I}$ is the set of connected components of $Z_1-Z_2$.
The group $S_k$ acts on $\{ C_I^j\}_{j,I}$. For fixed  $j$,  the subset $\{C_I^j\}_{I}$
 forms a single orbit.

\smallskip

For $p\geq 2$, the connected components $\{ C_{\bs I}^{\bs j}\}_{\bs j,\bs I}$ of $Z_p-Z_{p+1}$  are labeled by pairs $(\bs j, \bs I)$,
 where $\bs j$ is a $p$-element subset of $\{1,\dots,n_0\}$
and $\bs I =\{I_j\}_{j\in\bs j}$ is a set of pairwise disjoint subsets of $\{1, \dots, k\}$
such that $|I_j|=m_{j}+1$.  The connected component
$C_{\bs I}^{\bs j}$ is isomorphic  to the complement of
the product of weighted arrangements in $(\times_{j\in \bs j} \mc P^{m_{j}})\times \C^{e(\bs j)}$,
 where $e(\bs j) = k-p-\sum_{j\in\bs j} m_{j}$.
For $j\in \bs j$, if $I_j=\{1\leq i_0<\dots<i_{m_j}\leq k\}$, then  $u_{i_l}$, $l=1,\dots,m_j$,
are coordinates on an affine chart on $\mc P^{m_j}$, see Section \ref{sec localproj}.
 The arrangement in
$\mc P^{m_j}$ has hyperplanes $u_{i_{l}}=0, u_{i_{l}}-1=0, u_{i_{l}}-u_{i_{q}}=0$ and the hyperplane at infinity.
 The weights induced by $\mc A$
are $-m_j/\ka$ for $u_{i_l}=0$ and $2/\ka$ for $u_{i_l}-1=0$ and $u_{i_l}-u_{i_q}=0$.
   The space $\C^{e(\bs j)}$ has coordinates $t_i$,
$i \in \{1,\dots,k\} - \cup_{j\in J}I_j$. The weighted arrangement in $\C^{e(\bs j)}$ is the discriminantal arrangement
with weights given given by the pair $(\bs m^{\bs j},\ka)$, where  $m_i^{\bs j}=-m_i-2$ if $i\in\bs j$  and $m_i^{\bs j}=m_i$ otherwise,
see Section \ref{sec discarr}.

The group $S_k$ acts on the set $\{ C_{\bs I}^{\bs j}\}_{\bs j,\bs I}$. For fixed  $\bs j$,  the subset  $\{ C_{\bs I}^{\bs j}\}_{\bs I}$
forms a single orbit.

\smallskip
Let $ C_{\bs I}^{\bs j}$ be a connected component of $Z_p-Z_{p+1}$ and
$ C_{\tilde{\bs I}}^{\tilde {\bs j}}$ a connected component of $Z_{p+1}-Z_{p+2}$. Then
$ C_{\tilde{\bs I}}^{\tilde {\bs j}}$ lies in the closure of $ C_{\bs I}^{\bs j}$ if and only if
 $\bs j\subset \tilde{\bs j}$ and  $I_j = \tilde I_j$ for every $j\in {\bs  j}$.

\subsection{Skew-symmetric forms on weighted Orlik-Solomon manifold}
\label{skeW forms}

For $p>0$, fix a set $\bs j=\{1\leq j_1<\dots<j_p\leq n_0\}$. Consider the
 $S_k$-orbit $\{ C_{\bs I}^{\bs j}\}_{\bs I}$ of connected components of $Z_p-Z_{p+1}$.
Recall that  $\bs I =\{I_j\}_{j\in\bs j}$ is a set of pairwise disjoint subsets of $\{1, \dots, k\}$
such that $|I_j|=m_{j}+1$.
Each component $C_{\bs I}^{\bs j}$ is invariant with respect to the action of the subgroup
$S_{\bs I} = S_{m_{j_1}+1}\times \dots\times S_{m_{j_p}+1}\times S_{e(\bs j)}\subset S_k$, where
$S_{m_{j}+1}$ is the group of permutations of elements of the subset $I_j$,
$e(\bs j)=k-p-\sum_{l=1}^p m_{j_l}$ and $S_{e(\bs j)}$ is the group of permutations of elements of the subset $\{1,\dots,k\}-\cup_{j\in\bs j} I_j$.

Our goal is to describe $S_k$-skew-symmetric Orlik-Solomon forms on $\cup_{\bs I} C_{\bs I}^{\bs j}$.
 Such a form is uniquely determined by its restriction to one of the components $\{C_{\bs I}^{\bs j}\}_{\bs I}$.
That restriction is   $S_{\bs I}$-skew-symmetric.
According to Sections \ref{sec twodim} and \ref{sec localproj}, the $S_k$-skew-symmetric Orlik-Solomon forms on $\cup_{\bs I} C_{\bs I}^{\bs j}$
are available only in degrees $k-p$ and $k-p-1$.

Denote
\bea
d_{\bs j} =\sum_{i=1}^{p-1}i(m_{j_i}+1), \qquad
 s_{\bs j} = p+\sum_{i=1}^{p}m_{j_i}.
 \eea
Select in $\{ C_{\bs I}^{\bs j}\}_{\bs I}$ the component $C^{\bs j}_{\bs I^{0}}$, where $\bs I^0 =\{I_{j_1}^0,\dots,I_{j_p}^0\}$
and
\bea
&&
I^0_{j_i}=\{1+\sum_{l=1}^{i-1}(m_{j_l}+1),\dots,m_{j_i}+1+\sum_{l=1}^{i-1}(m_{j_l}+1)\} ,
\qquad
i=1,\dots, p.
\eea
Let $K=(k_1,\dots, k_n)\in\Z_{\geq 0}^n$, where $|K|$ equals $e(\bs j)$ or $e(\bs j)-1$. Denote $l_0(K)=0$ and $l_i(K)=k_1+\dots+k_i$, $i=1,\dots,n$.
 Denote
\bea
\eta_{K,i}^{\bs j} = d\log (t_{s_{\bs j}+l_{i-1}+1}-z_i)\wedge d\log (t_{s_{\bs j}+l_{i-1}+2}-z_i)\wedge\dots\wedge d\log (t_{s_{\bs j}+l_{i}}-z_i),
\qquad
i=1,\dots,n,
\eea
\bea
\al_{K}^{\bs j}=(-1)^{d_{\bs j}}((m_{j_1}+1)!\dots(m_{j_p}+1)!\,k_1!\dots k_n!)^{-1} .
\eea
The form
$
\al_K^{\bs j}\, \mu_{I^0_{j_1}}\wedge\dots\wedge\mu_{I^0_{j_p}}\wedge  \eta_{K,1}^{\bs j}\wedge\dots\wedge \eta_{K,n}^{\bs j}
$
 is an Orlik-Solomon form
on  $C^{\bs j}_{\bs I^{0}}$. We extend it by zero  to other components of
 $\cup_{\bs I}C^{\bs j}_{\bs I}$. If $|K|=e(\bs j)$, we define the form $\omega_{K}^{\bs j}$ on $\cup_{\bs I}C^{\bs j}_{\bs I}$ as
 the $S_{k}$-skew-symmetrization of the form  $\ka^{-k}
\al_K^{\bs j}\, \mu_{I^0_{j_1}}\wedge\dots\wedge\mu_{I^0_{j_p}} \wedge \eta_{K,1}^{\bs j}\wedge\dots\wedge \eta_{K,n}^{\bs j}
$.
If $|K|=e(\bs j)-1$, we define the from $\omega_{K}^{\bs j}$ on $\cup_{\bs I}C^{\bs j}_{\bs I}$ as
the $S_{k}$-skew-symmetrization of the form
$(-1)^{k-p} \ka^{k-1}\al_K^{\bs j}\, \mu_{I^0_{j_1}}\wedge\dots\wedge\mu_{I^0_{j_p}} \wedge \eta_{K,1}^{\bs j}\wedge\dots\wedge \eta_{K,n}^{\bs j}$.

Denote by $A^\bullet_-(\cup_{\bs I}C^{\bs j}_{\bs I})\subset \oplus_{\bs I}A^\bullet(C^{\bs j}_{\bs I})$
the skew-symmetric part of the Orlik-Solomon space $\oplus_{\bs I}A^\bullet(C^{\bs j}_{\bs I})$ of $\cup_{\bs I}C^{\bs j}_{\bs I}$. Recall
the 1-form $\om_a$ in \Ref{oma}. The form $\om_a$ lifts to an element $\tilde \om_a=\pi^*\om_a$ of
$\oplus_{\bs j, \bs I}A^1(C^{\bs j}_{\bs I})$ which is symmetric with respect to the $S_k$ action.
The exterior multiplication by $\tilde \om_a$ defines
the complex
\bean
\label{comp bs j}
\wedge \tilde \om _a \ :\ A^{k-p-1}_-(\cup_{\bs I}C^{\bs j}_{\bs I})\ \to\ A^{k-p}_-(\cup_{\bs I}C^{\bs j}_{\bs I}).
\eean
Recall the vector $\bs m^{\bs j}=(m_1^{\bs j},\dots,m_n^{\bs j})$ from Section \ref{associated weighted Orlik-Solomon}.

\begin{lem}
\label{lem skew on OS}

The complex in \Ref{comp bs j} is isomorphic
to the weight component of weight $|\bs m|-2k$ of the complex $\n_-\otimes (\otimes_{i=1}^n M_{m_i^{\bs j}})^* \to (\otimes_{i=1}^n M_{m_i^{\bs j}})^*$,
see \Ref{C(n,V)}.
The isomorphism sends $\om_K^{\bs j}$ to $(-1)^p f\otimes \phi^K_{\bs m^{\bs j}} $ if $|K|=e(\bs j)-1$ and to $\phi^K_{\bs m^{\bs j}}$ if $|K|=e(\bs j)$.
\end{lem}

Lemma \ref{lem skew on OS} is a corollary of Lemma \ref{thm SV}.

\subsection{Residues of skew-symmetric forms}
\label{residues}

Consider an  $S_k$-orbit $\{ C_{\bs I}^{\bs j}\}_{\bs I}$ of connected components of $Z_p-Z_{p+1}$
and an $S_k$-orbit $\{ C_{\tilde{\bs I}}^{\tilde{\bs j}}\}_{\tilde{\bs I}}$ of connected components of $Z_{p+1}-Z_{p+2}$
such that the second orbit lies in the closure of the first orbit. This statement holds if and only if $\bs j\subset \tilde {\bs j}$.
More precisely,  if
 $\bs j=\{j_1<\dots<j_p\}$, then $\tilde{\bs j}=\{j_1<\dots<j_q < \tilde j_{q+1}  <j_{q+1}<\dots< j_p\}$
for some $0\leq q\leq p$.

Consider $\omega_{K}^{\bs j}\in A^\bullet_-(\cup_{\bs I}C^{\bs j}_{\bs I})$.
Then the residue
of $\omega_{K}^{\bs j}$ at $\cup_{\tilde{\bs I}} C_{\tilde{\bs I}}^{\tilde{\bs j}}$  is
an element of $A^\bullet_-(\cup_{\tilde{\bs I}}C^{\tilde{\bs j}}_{\tilde{\bs I}})$.
We denote this residue by $\res_{\bs j}^{ \tilde{\bs j}}\omega_{K}^{\bs j}$.

\begin{lem}
\label{res lem}
Given $K=(k_1,\dots,k_n)$, denote $\tilde K=(k_1,\dots,k_{\tilde j_{q+1}}\!\!-m_{\tilde j_{q+1}}\!\!-1,\dots,k_n)$.
If $k_{\tilde j_{q+1}} < m_{\tilde j_{q+1}}+1$, then $\res_{\bs j}^{ \tilde{\bs j}}\omega_{K}^{\bs j}=0$.
If $k_{\tilde j_{q+1}}\geq m_{\tilde j_{q+1}}+1$, then $\res_{\bs j}^{ \tilde{\bs j}}\omega_{K}^{\bs j} = (-1)^q \omega_{\tilde K}^{\bs j}$ for $|K|=e(\bs j)$
and  $\res_{\bs j}^{ \tilde{\bs j}}\omega_{K}^{\bs j} = (-1)^{q+1} \omega_{\tilde K}^{\bs j}$ for $|K|=e(\bs j)-1$.

\end{lem}

\begin{proof}
If $k_{\tilde j_{q+1}} < m_{\tilde j_{q+1}}+1$, then the form $\omega_{K}^{\bs j}$ is regular on $\cup_{\tilde{\bs I}}C^{\tilde{\bs j}}_{\tilde{\bs I}}$
and $\res_{\bs j}^{ \tilde{\bs j}}\omega_{K}^{\bs j}=0$.
If  $k_{\tilde j_{q+1}}\geq m_{\tilde j_{q+1}}+1$, then the statement is checked by direct calculation.
\end{proof}

\subsection{Skew-symmetric part of  Aomoto complex of weighted Orlik-Solomon manifold}
\label{skEW Ao}
Consider the weighted Orlik-Solomon manifold $(X,D)$ introduced in Section \ref{associated weighted Orlik-Solomon}
and its Aomoto complex $(A^\bullet(X,\mc Z),\tilde\om_a+\res)$ introduced  in Section \ref{Weighted discr arrangement}.
By Theorem \ref{thm 4}, for generic nonzero $\ka$ the complex $(A^\bullet(X,\mc Z),\tilde\om_a+\res)$ calculates the
cohomology  $H^\bullet(X,\L_{\tilde\om_a})$
of $X$ with coefficients in the rank 1 local system $\L_{\tilde\om_a}$ on $X$ associated with the differential
form $\tilde \om_a$, see Corollary \ref{cor generic}.

The group $S_k$ acts on the complex. Denote $(A^\bullet_-(X,\mc Z),\tilde\om_a+\res)$ the skew-symmetric part of the
complex. For generic nonzero $\ka$ the complex $(A^\bullet_-(X,\mc Z),\tilde\om_a+\res)$ calculates the
skew-symmetric part $H^\bullet_-(X,\L_{\tilde\om_a})$ of the  cohomology  $H^\bullet(X,\L_{\tilde\om_a})$.

Recall  the complex $(B^\bullet_\mm[|\bs m|-2k], \tilde d)$ in Section \ref{Tensor product of complexes}.
Define the linear map
\bean
\label{Iso}
\ga &:& A^\bullet_-(X,\mc Z)\to B^\bullet_\mm[|\bs m|-2k],
\\
\notag
&&
 \om^{\bs j}_K\mapsto f\otimes \phi^{K}_{\bs m^{\bs j}}\ \on{if}\ |K|=e(\bs j)-1, \quad
  \om^{\bs j}_K\mapsto \phi^{K}_{\bs m^{\bs j}}\ \on{if}\ |K|=e(\bs j).
\eean

\begin{thm}
\label{thm skew main}
The map $\ga$ defines the isomorphism of complexes
$(A^\bullet_-(X,\mc Z),\tilde\om_a+\res)$ and $(B^\bullet_\mm [|\bs m|-2k], \tilde d)$.

\end{thm}

\begin{proof}
The theorem follows from Lemmas\ \ref{lem skew on OS} and  \ref{res lem}.
\end{proof}

The quasi-isomorphism $C_\bullet(\n_-,\otimes_{a=1}^nL_{m_a})[|\bs m|-2k] \to (B^\bullet_\mm,\tilde d)[|\bs m|-2k]$ in \Ref{bgg}
allows us to identify the cohomology  $H^\bullet_-(X,\L_{\tilde\om_a})$ and the cohomology of the complex
\linebreak
 $C_\bullet(\n_-,\otimes_{a=1}^nL_{m_a})[|\bs m|-2k]$. Namely, let
 \bea
\otimes_{a=1}^nL_{m_a} = \oplus_{p} L_p\otimes W_p
\eea
be the decomposition of the tensor product into irreducible $\slt$-modules, where $W_p$ are multiplicity spaces.

\begin{cor}
\label{cor dim}
If $|\bs m|-2k\geq 0$, then
\bea
\dim H^k_-(X,\L_{\tilde\om_a}) = \dim W_{|\bs m|-2k} \quad \on{and}\quad   H^q_-(X,\L_{\tilde\om_a}) = 0\quad \on{for}\quad  q\ne k.
\eea
If $|\bs m|-2k =-1$, then $H^{\bullet}_-(X,\L_{\tilde\om_a})=0$. If $|\bs m|-2k <-1$, then
\bea
\dim H^{k-1}_-(X,\L_{\tilde\om_a}) = \dim W_{2k-2-|\bs m|}
\quad \on{and}
\quad   H^q_-(X,\L_{\tilde\om_a}) = 0 \quad
\on{ for} q\ne k-1.
\eea
\end{cor}

\subsection{BGG resolution and flag forms}
\label{BGG resolution and flag forms}

Theorem \ref{thm skew main} gives a geometric interpretation of the BGG resolution given in \Ref{bgg}.
Namely, the embeddings
\bea
\n_-\otimes\otimes_{a=1}^nL_{m_a}[|\bs m|-2k+2]  &\hookrightarrow & \n_-\otimes(\OT)^*[|\bs m|-2k+2],
\\
\otimes_{a=1}^nL_{m_a}[|\bs m|-2k]  & \hookrightarrow & (\OT )^*[|\bs m|-2k]
\eea
in \Ref{emb} have the form:\  the element $f\otimes f^Kv_{\bs m}$ is mapped to $\beta^K_{\bs m}\,f\otimes \phi^K_{\bs m}$ if
$|K|=k-1$ and  the element $f^Kv_{\bs m}$ is mapped to $\beta^K_{\bs m} \phi^K_{\bs m}$ if
$|K|=k$, where
\bea
\beta^K_{\bs m} = \prod_{i=1}^n k_i! \prod_{\ell=1}^{k_i}(m_i+1-\ell).
\eea
Under the isomorphism of Theorem \ref{thm skew main}, we obtain embeddings
\bea
\n_-\otimes\otimes_{a=1}^nL_{m_a}[|\bs m|-2k+2]  \hookrightarrow  A^{k-1}_-(U), &\quad & f\otimes f^Kv_{\bs m}\mapsto \beta^K_{\bs m}\,\om_{K},
\\
\otimes_{a=1}^nL_{m_a}[|\bs m|-2k]   \hookrightarrow  A^k_-(U),  &\quad & f^Kv_{\bs m}\mapsto \beta^K_{\bs m}\,\om_{K} .
\eea
The images
\bea
\mc F^{k-1}_-=\on{span}\langle \beta^K_{\bs m}\om_K\rangle_{|K|=k-1}\subset  A^{k-1}_-(U),
\qquad
\mc F^{k}_-=\on{span}\langle \beta^K_{\bs m}\om_K\rangle_{|K|=k}\subset  A^{k}_-(U)
\eea
 of these embeddings are called  the {\it subspaces of skew-symmetric flag forms}, see \cite{SV, V}. The exterior multiplication
 by $\om_a$ gives the {\it complex of skew-symmetric flag forms} $\wedge \om_a : \mc F^{k-1}_- \to \mc F^{k}_-$.
 Now the BGG resolution  in \Ref{bgg} can be interpreted as the statement that the natural embedding of the complex of skew-symmetric flag
 forms to the complex $(A^\bullet_-(X,\mc Z),\tilde\om_a+\res)$ is a quasi-isomorphism.

 The complex of  skew-symmetric flag forms can be characterized as follows.

 \begin{lem}
 \label{lem falg res}
The vector space $\mc F^{\bullet}_-$ is the kernel of the residue map
\bea
A^{\bullet}_-(U) \to  \oplus_{j=1}^{n_0} A^\bullet_-(\cup_{|I|=m_j+1}C^j_I).
\eea
 \end{lem}

\begin{proof}
The lemma follows from Lemma  \ref{res lem}.
\end{proof}

\subsection{Cohomology $H^\bullet(X,\L_{\tilde\om_a})$ and intersection cohomology}
\label{Intersec}

Let $j : U\to \C^k$ be the canonical embedding. Let $\L_{\om_a}$ be the rank 1 local system on $U$
associated with the form $\om_a$, see Section \ref{Weighted discr arrangement}.
Consider the   intersection cohomology  $H^\bullet(\C^k, j_{!*}\L_{\om_a})$. By \cite{AV}, for generic nonzero real $\ka$,
the intersection cohomology
 $H^\bullet(\C^k, j_{!*}\L_{\om_a})$ is canonically isomorphic to the cohomology
  $H^\bullet(X, \L_{\tilde \om_a})$ if the following condition A from \cite{AV} is satisfied.

For $1\leq j\leq n_0$, consider $\C^{m_j}$ with coordinates $u_1,\dots, u_{m_j}$. Consider  the weighted
arrangement in $\C^{m_j}$ consisting of the hyperplanes $u_i=0, u_i-1=0$, $u_i-u_p=0$ with weights $-m_j/\ka$ for hyperplanes $u_i=0$
and weights $2/\ka$ for hyperplanes $u_i-1=0$ and $u_i-u_p=0$, c.f. Section \ref{associated weighted Orlik-Solomon}.
Denote by $U_j\subset \C^{m_j}$ the complement to the union
of hyperplanes of the arrangement. Let $\L_j$ be the rank 1 local system on $U_j$ associated with this weighted arrangement,
see Section \ref{Weighted discr arrangement}. The condition A is satisfied if for any $1\leq j\leq n_0$ we have
$H^\ell(U_j,\L_j)=0$ for $\ell > m_j$. Clearly in this situation condition A is satisfied  and
$H^\bullet(\C^k, j_{!*}\L_{\om_a})$ is canonically isomorphic to the cohomology
$H^\bullet(X, \L_{\tilde \om_a})$ by \cite{AV}. In particular, this implies that for    generic nonzero real $\ka$,
the skew-symmetric part $H^\bullet_-(\C^k, j_{!*}\L_{\om_a})$ on  the intersection cohomology  $H^\bullet(\C^k, j_{!*}\L_{\om_a})$
is isomorphic to the cohomology of the complex $(A^\bullet_-(X,\mc Z),\tilde\om_a+\res)$  and, hence,  to the cohomology of the complex
 $C_\bullet(\n_-,\otimes L_{m_j}) [|\bs m|-2k]$, see Section \ref{skEW Ao},
 c.f.
 \linebreak
 \cite[Section 6 of Introduction]{KV} and \cite[Corollary 6.11]{KV}.

\subsection{Remark} In the constructions of Section \ref{sec slT} we may assume
that  $\bs m=(m_1,\dots, m_n)$ is a vector with arbitrary complex coordinates
instead of being a vector with positive integer coordinates. Then all statements of Section \ref{sec slT}
hold. In particular, the same proofs show that in this more general situation the complex  $C_\bullet(\n_-,\otimes_{a=1}^nL_{m_a})[|\bs m|-2k]$
calculates the cohomology  $H^\bullet_-(X,\L_{\tilde\om_a})$ as well as the intersection cohomology $H^\bullet_-(\C^k, j_{!*}\L_{\om_a})$.


\end{document}